\documentclass[11pt,twoside,reqno,centertags]{amsart}
\usepackage{amsmath,amsthm,amsfonts,amssymb}
\advance\textheight by 6truemm
\newtheorem{Theorem}{Theorem}
\newtheorem{Lemma}[Theorem]{Lemma}
\newtheorem{theorem}[Theorem]{Theorem}
\newtheorem{corollary}[Theorem]{Corollary}
\newtheorem{proposition}[Theorem]{Proposition}
\newtheorem{lemma}[Theorem]{Lemma}
\newtheorem{remark}[Theorem]{Remark}
\newcommand{\R}{{\mathbb R}}
\newcommand{\eps}{\varepsilon}
\font\basic=cmr10
\setbox1=\hbox{\basic\strut Department of Applied Mathematics and Statistics, Comenius University} 
\setbox2=\hbox{\basic\strut Mlynsk\'a dolina, 84248 Bratislava, Slovakia}
\setbox3=\hbox{\basic\strut email: {\tt quittner@fmph.uniba.sk}}

\begin{document}

\title[Liouville theorems for systems]{Liouville theorems for parabolic systems \\ 
       with homogeneous nonlinearities \\ and gradient structure}
\dedicatory{Dedicated to Eiji Yanagida on the occasion of his 65th birthday}
\author{Pavol Quittner \\ \\ \box1 \\ \box2 \\ \box3}
\thanks{Supported in part by the Slovak Research and Development Agency 
        under the contract No. APVV-18-0308 and by VEGA grant 1/0339/21.} 
\date{}

\begin{abstract}
Liouville theorems for scaling invariant nonlinear parabolic
equations and systems (saying that the equation or system
does not possess nontrivial entire solutions)  
guarantee optimal universal estimates of solutions of
related initial and initial-boundary value problems.
Assume that $p>1$ is subcritical in the Sobolev sense.
In the case of nonnegative solutions and the system
$$U_t-\Delta U=F(U)\quad\hbox{in}\quad \R^n\times\R,$$ 
where $U=(u_1,\dots,u_N)$, $F=\nabla G$ is $p$-homogeneous
and satisfies the positivity assumptions $G(U)>0$ for $U\neq0$ and
$\xi\cdot F(U)>0$ for some $\xi\in\R^N$ and all $U\geq0$, $U\ne 0$,
it has recently been shown in  
[P. Quittner, {\it Duke Math. J.} 170 (2021), 1113--1136]
that the parabolic Liouville theorem
is true whenever the corresponding elliptic Liouville theorem
for the system $-\Delta U=F(U)$ is true.
By modifying the arguments in that proof we show
that the same result remains true without the positivity assumptions on $G$ and $F$,
and that the class of solutions can also be enlarged to contain
(some or all) sign-changing solutions.
In particular, in the scalar case $N=1$ and $F(u)=|u|^{p-1}u$,
our results cover the main result in 
[T. Bartsch, P. Pol\'a\v cik and P. Quittner, {\it J. European Math. Soc.} 13 (2011), 219--247].
We also prove a parabolic Liouville theorem for solutions in $\R^n_+\times\R$
satisfying homogeneous Dirichlet boundary conditions on $\partial\R^n_+\times\R$
since such theorem is also needed if one wants to prove universal
estimates of solutions of related systems in $\Omega\times(0,T)$,
where $\Omega\subset\R^n$ is a smooth domain.
Finally, we use our Liouville theorems
to prove universal estimates for particular parabolic systems.


\end{abstract}

\maketitle

\vfill\eject

\section{Introduction}
\label{intro}

Similarly as in the elliptic case,
Liouville theorems for scaling invariant nonlinear parabolic
equations and systems (saying that the equation or system
does not possess nontrivial entire solutions)  
guarantee optimal universal estimates of solutions of
related initial and initial-boundary value problems,
including singularity and decay estimates
(see, for example, the introduction in \cite{Q-DMJ}
or \cite{PQS1,PQS} for more details).
Unfortunately, many methods used in the proofs
of elliptic Liouville theorems cannot be used
or do not yield optimal results in the parabolic case
(compare \cite{GS} with \cite{BV}, for example).

If a parabolic problem has the gradient structure,
then it often possesses a Lyapunov functional
(also called energy) which is decreasing
along nonstationary solutions and which guarantees that
the $\alpha$- and $\omega$-limit sets of entire solutions
consist of equilibria.
In addition, if an entire solution is not stationary,
then the energy of any stationary solution belonging to
its $\alpha$-limit set has to be greater than the energy
of any stationary solution belonging to its $\omega$-limit set.
This suggests that the nonexistence
of nontrivial equilibria should imply the nonexistence of   
nontrivial entire solutions. However,
this simple argument cannot be used straightforwardly.
The problem is that the energy is only defined
for a restricted class of functions
and the universal estimates mentioned above require
the nonexistence of entire solution in a larger class.
For example, 
in the case of the model scalar equation
\begin{equation} \label{Fuj}
 u_t-\Delta u=|u|^{p-1}u,  \quad (x,t)\in\R^n\times\R,
\end{equation}
with $p>1$, the energy has the form
\begin{equation} \label{Ev}
 E(v)=\frac12\int_{\R^n}|\nabla v|^2\,dx-\frac1{p+1}\int_{\R^n}|v|^{p+1}\,dx,
\end{equation}
hence the energy argument only applies to solutions $u$ such that the integrals
in \eqref{Ev} are finite for $v:=u(\cdot,t)$, and a typical useful Liouville
theorem has to hold true for all positive bounded solutions of \eqref{Fuj}, for example.
In fact, if $n>10$ and $p$ is sufficiently large, then there exist 
positive bounded classical solutions of \eqref{Fuj} which are
homoclinic to the zero solution (see \cite{FY}) and this fact
shows that the energy cannot be used 
to rule out the existence of nontrivial entire solutions in this case.

An efficient method of using energy in the proof of 
parabolic Liouville theorems has recently been introduced
in \cite{Q-DMJ}. In that paper, the author used an
energy functional associated to a rescaled problem
(where the energy is defined for all bounded solutions)
and combined arguments in \cite{GK} with
bootstrap, doubling and measure arguments in order to show
that the nonexistence of positive equilibria indeed
guarantees the nonexistence of positive entire solutions
of the parabolic problem. More precisely, assuming that
the growth of the nonlinearity is Sobolev subcritical, he 
proved a Liouville theorem for positive solutions of the system
\begin{equation} \label{eq-U}
U_t-\Delta U=F(U)\quad\hbox{in}\quad \R^n\times\R,
\end{equation}
where $U=(u_1,\dots,u_N)$, $F=\nabla G:\R^N\to\R^N$ is $p$-homogeneous
and satisfies the positivity assumptions $G(U)>0$ for $U\neq0$ and
$\xi\cdot F(U)>0$ for some $\xi\in\R^N$ and all $U\geq0$, $U\ne 0$.

By modifying arguments in the proof in \cite{Q-DMJ} (and in the related paper \cite{Q-JDDE}) 
we show that the result in \cite{Q-DMJ} remains true without the positivity assumptions on
$G$ and $F$,
and that the class of solutions can also be enlarged to contain
(some or all) sign-changing solutions.
In particular, in the scalar case $N=1$ and $F(u)=|u|^{p-1}u$,
our result covers the main result for nodal solutions in \cite{BPQ}
which was proved by completely different arguments.

Let us explain the main difference between our proof and the proof in \cite{Q-DMJ}. 
The positivity of $U$ and $F$ enable one to use Kaplan-type estimates 
which guarantee that the energy of suitably rescaled functions $W_k$
at suitable times $s_k$, $k=1,2,\dots$, can be estimated above by $Ck^\beta$ with $\beta:=1/(p-1)$.
This simple result combined with the arguments in \cite{GK}   
already proves the claim if $n\leq 2$, see~\cite{Q-MA}.
If $2<n\leq 6$, then a relatively simple bootstrap argument 
in \cite{Q-DMJ} enables one to decrease the upper bound in the
energy estimate to $Ck^\gamma$, where $\gamma<\beta$
is small enough to use the arguments in \cite{GK} again.
If $n>6$, then such improved energy estimate was obtained in \cite{Q-DMJ}
by a more sophisticated bootstrap argument.
In this paper, instead of the initial estimate $Ck^\beta$
we just use a rough estimate $Ck^{(p+1)\beta}$
which does not require the positivity of $U$ and $F$.
This rough estimate and the arguments in \cite{GK}
do not guarantee the desired result for any $p>1$,
and the simpler bootstrap argument mentioned above
can only be used if $n=1$. On the other hand,
the more sophisticated bootstrap argument can be 
used in the full subcritical range again
and the proof is even simpler
(since one does not need to control the integrals
appearing in the Kaplan-type estimates).
It should be mentioned that many parts of our proof 
are almost identical to those in \cite{Q-DMJ,Q-JDDE}
but we provide them in detail in order to have a self-contained proof.

We also prove a parabolic Liouville theorem for solutions in the halfspace $\R^n_+\times\R$
satisfying homogeneous Dirichlet boundary conditions on $\partial\R^n_+\times\R$
since such theorem is also needed if one wants to prove universal
estimates of solutions of related systems in $\Omega\times(0,T)$,
where $\Omega\subset\R^n$ is a smooth domain, $\Omega\ne\R^n$.
Since the corresponding proof is again just a slight modification
of the proof of our main result, we just give its sketch. 

As an application of our results we prove universal estimates 
for particular systems of the form
\begin{equation} \label{eq-Upert}
U_t-\Delta U=F(U)+\tilde F(U)\quad\hbox{in}\quad \Omega\times(0,T),
\end{equation}
where $\tilde F$ is a perturbation term, $T\in(0,\infty]$ and $\Omega$ is a smooth
domain in $\R^n$. If $\Omega\ne\R^n$, then system \eqref{eq-Upert}
is complemented by homogeneous Dirichlet boundary conditions on $\partial\Omega\times(0,T)$.

If $N=2$, $n\le3$, $\Omega=\R^n$ or $\Omega=B_R:=\{x\in\R^n:|x|<R\}$, and
\begin{equation} \label{F}
F(U)=(|u_1|^{2q+2}u_1+\beta|u_2|^{q+2}|u_1|^q u_1,
      |u_2|^{2q+2}u_2+\beta|u_1|^{q+2}|u_2|^q u_2),
\end{equation}
where $\beta\ne-1$ and $p:=2q+3>1$ is Sobolev subcritical,
then we prove universal estimates for radial nodal solutions of \eqref{eq-Upert}
with finite zero number and finite number of intersections
(see Section~\ref{appl} and \cite{Q-LSE} for more details and motivation of
this study, respectively).

If $n\le5$ and either $N=1$, $F(u)=u^2$,
or $N=2$, $F(U)=(2u_1u_2,u_1^2+u_2^2)$,
then we prove universal estimates for all
(possibly sign-changing and nonradial) solutions of \eqref{eq-Upert}.

\section{Main results}
\label{main}

Let us introduce some notation first.
By $p_S$ we denote the critical Sobolev exponent:
$$
p_S:=\left\{
  \begin{aligned}
    &\frac{n+2}{n-2},&\quad &\text{if $n\ge 3$,}\\
    &\infty,&\quad &\text{if $n\in\{1,2\}$.}
  \end{aligned}\right.
$$
By a nontrivial solution $U$ we understand
a solution $U=(u_1,\dots,u_N)\not\equiv(0,\dots,0)$.
By $|U(x,t)|$ (or $|\nabla U(x,t)|$, resp.) we denote the Euclidean norm
of $U(x,t)\in\R^N$ (or $\nabla U(x,t)\in\R^{nN}$, resp.);
the partial derivative $\frac{\partial U}{\partial t}$
is sometimes also denoted by $U_t$.
In the proofs we also use the notation $|A|$ to denote
the Lebesgue measure of a set $A\subset\R^n$,
and by $\#A$ we denote the cardinality of an arbitrary set $A$.

If $J\subset \R$ is an interval and $v\in C(J,\R)$,
then we define 
$$ \begin{aligned}
  z(v)=z_J(v):=\sup\{j:&\,\exists x_1,\dots,x_{j+1}\in J,\
     x_1<x_2<\dots<x_{j+1},\\ 
    &v(x_i)\cdot v(x_{i+1})<0 \hbox{ for } i=1,2,\dots,j\},
\end{aligned}
$$
where $\sup(\emptyset):=0$. We usually refer to $z_J(v)$  as the \emph{zero
number} of $v$ in $J$. Note that $z_J(v)$ is actually the number of
sign changes of $v$; it coincides with the number of zeros of $v$ if 
$J$ is open, $v\in C^1(J)$ and all its zeros are simple. 
If $v:\R^n\to\R$ is a continuous, radially symmetric function,
i.e. $v(x)=\tilde v(|x|)$ for some $\tilde v\in C([0,\infty),\R)$,
then we set $z(v):=z(\tilde v)$.

Let $I\subset\{1,2,\dots,N\}$ ($I$ may be empty) and
$${\mathcal K_I}:=\{U\in C^2: u_i\geq0\hbox{ for }i\in I\}.$$
If $n=1$ (or $n>1$ and we consider radial solutions),  
$K\in\{0,1,2,\dots\}$, $\alpha_{ij}\in\R$ and $C_j\geq0$ ($i=1,2,\dots,N,\ j=1,2,\dots,K$), then
we also set
$${\mathcal K}_z={\mathcal K}_z(\alpha_{ij},C_j,I):=\Bigl\{U\in{\mathcal K_I}: 
  z\Bigl(\sum_{i=1}^N\alpha_{ij}u_i\Bigr)\leq C_j,\ j=1,2,\dots,K\Bigr\}.$$

In the case of the parabolic system \eqref{eq-U}
our main result is the following theorem.

\begin{theorem} \label{thmU}
Let $1<p<p_S$ and let either $n>1$ and ${\mathcal K}={\mathcal K}_I$,
or $n=1$ and ${\mathcal K}\in\{{\mathcal K}_I,{\mathcal K}_z\}$.
Assume 
\begin{equation} \label{FG}
F=\nabla G,\ \hbox{ with }\ 
G\in C^{1+\alpha}_{loc}(\R^N,\R)\ \hbox{ for some }\ \alpha>0,\ G(0)=0,
\end{equation} 
\begin{equation} \label{F1}
F(\lambda U)=\lambda^p F(U)\quad\hbox{for all }\ \lambda\in(0,\infty)
                    \hbox{ and all }\ U\in\R^N.
\end{equation}
If system \hbox{\rm(\ref{eq-U})} does not possess
nontrivial classical stationary solutions in ${\mathcal K}$, then
it does not possess any nontrivial classical solution 
satisfying $U(\cdot,t)\in{\mathcal K}$ for all $t\in\R$.
\end{theorem}

Using the same arguments as in the proof of \cite[Proposition 2.4]{BPQ}
one can easily show that Theorem~\ref{thmU} with $n=1$
implies the following result:  

\begin{corollary} \label{cor}
Let $1<p<p_S$ and $n>1$. 
Assume \eqref{FG} and \eqref{F1}.
Assume also that system \hbox{\rm(\ref{eq-U})} does not possess
nontrivial classical stationary radial solutions
in ${\mathcal K}_z$ and that system \hbox{\rm(\ref{eq-U})} with $n=1$
does not possess nontrivial classical stationary solutions in ${\mathcal K}_z$.
Then system \hbox{\rm(\ref{eq-U})} (with $n>1$) does not possess
any nontrivial classical radial solution satisfying $U(\cdot,t)\in{\mathcal K}_z$ for all $t$.
\end{corollary}  

In the case of the halfspace $\R^n_+:=\{x=(x_1,x_2,\dots,x_n)\in\R^n:x_1>0\}$
and solutions of the problem
\begin{equation} \label{eq-Uhalf}
\begin{aligned}
 U_t-\Delta U &=F(U), &\qquad& x\in\R^n_+,\ t\in\R, \\
 U &=0, &\qquad& x\in\partial\R^n_+,\ t\in\R,
\end{aligned}
\end{equation}
we have the following result:

\begin{theorem} \label{thmUhalf}
Let $1<p<p_S$ and let either $n>1$ and ${\mathcal K}={\mathcal K}_I$,
or $n=1$ and ${\mathcal K}\in\{{\mathcal K}_I,{\mathcal K}_z\}$.
Assume \eqref{FG} and \eqref{F1}.
If problems \hbox{\rm(\ref{eq-U})} and \hbox{\rm(\ref{eq-Uhalf})} do not possess
nontrivial classical stationary solutions in ${\mathcal K}$, then
problem \hbox{\rm(\ref{eq-Uhalf})}
does not possess any nontrivial classical solution 
satisfying $U(\cdot,t)\in{\mathcal K}$ for all $t\in\R$.
\end{theorem}

\section{Applications}
\label{appl}

As already mentioned, Liouville-type results of the form above
yield universal estimates of solutions of related initial-boundary value problems,
including optimal blowup rate and decay estimates,
see \cite{PQS,Q-DMJ}.
In this section we formulate such results for a few particular problems.

Assume first that
\begin{equation} \label{ass-Schr}
\left.\begin{aligned}
 &\Omega=B_R\ \hbox{ or }\ \Omega=\R^n,\quad n\leq3, \\
 &p:=2q+3\in(1,p_S),\ \beta,\lambda,\gamma\in\R, 
\end{aligned}\ \right\}
\end{equation}
and consider radial
solutions of the parabolic problem
\begin{equation} \label{Schr-lambda}
\begin{aligned}
&\left.
\begin{aligned}
u_t&=\Delta u-\lambda u-\gamma v+|u|^{p-1}u+\beta|v|^{q+2}|u|^qu, \\
v_t&=\Delta v-\lambda v-\gamma u+|v|^{p-1}v+\beta|u|^{q+2}|v|^qv,
\end{aligned}
\ \right\}
\ \hbox{ in } \Omega\times(0,T), \\
&\quad u=v=0 \quad\hbox{ on }\partial\Omega\times(0,T)\qquad(\hbox{if }\ \Omega=B_R).
\end{aligned}
\end{equation}
The study of \eqref{Schr-lambda} is motivated in \cite{Q-LSE}:
Our estimates can be used to prove the existence and multiplicity
of stationary solutions with prescribed nodal or intersection properties,
and such solutions correspond to solitary waves 
of corresponding systems of Schr\"odinger equations if $p=3$. 

The scaling invariant problem related to  \eqref{Schr-lambda} is
\begin{equation} \label{Schr-Liouv}
\left.
\begin{aligned}
u_t&=\Delta u+|u|^{p-1}u+\beta|v|^{q+2}|u|^qu, \\
v_t&=\Delta v+|v|^{p-1}v+\beta|u|^{q+2}|v|^qv,
\end{aligned}
\ \right\}
\ \hbox{ in } \R^n\times\R,
\end{equation}
and in the case $\Omega=B_R$ we also have to study the problem
\begin{equation} \label{Schr-bdry}
\begin{aligned}
&\left.
\begin{aligned}
u_t&= u_{xx}+|u|^{p-1}u+\beta|v|^{q+2}|u|^qu, \\
v_t&= v_{xx}+|v|^{p-1}v+\beta|u|^{q+2}|v|^qv,
\end{aligned}
\ \right\}
\ \hbox{ in } (0,\infty)\times\R, \\
&\quad u=v=0 \ \hbox{ on } \{0\}\times\R.
\end{aligned}
\end{equation}
Set 
\begin{equation} \label{K}
{\mathcal K}:=\{(u,v): z(u)\leq C_1,\ z(v)\leq C_2,\ 
          z(u-v)\leq C_3,\ z(u+v)\leq C_4\},
\end{equation}
\begin{equation} \label{Kplus}
{\mathcal K}^+:=\{(u,v): u,v\geq0,\ z(u-v)\leq C_3\},
\end{equation}
and notice that if $\gamma=0$, then ${\mathcal K}$ is invariant under the semiflow generated by \eqref{Schr-lambda},
since $u,v,u-v,u+v$ solve linear parabolic equations of the form $w_t-\Delta w=fw$, where $f=f(x,t)$.
Similarly, if $\gamma<0$, then ${\mathcal K}^+$ is invariant under the semiflow generated by \eqref{Schr-lambda}.

The following  elliptic Liouville theorem has recently been proved in
\cite{Q-LSE}. 
Notice also that if $\beta\in(-1,\infty)$ or $\beta>0$ and one considers nonnegative solutions, 
then the nonexistence of nontrivial (radial and nonradial) stationary solutions 
for problems appearing in the following theorem has been studied in
\cite{QS9,DW} or \cite{RZ}, respectively.

\begin{theorem} \label{thmSchrL}
Assume that $n\leq3$ and $p=2q+3\in(1,p_S)$. 
If $\beta\ne-1$, then system \eqref{Schr-Liouv} does not possess 
nontrivial classical radial stationary solutions satisfying $(u,v)\in{\mathcal K}$ 
and system \eqref{Schr-Liouv} with $n=1$ does not possess
nontrivial classical stationary solutions satisfying $(u,v)\in{\mathcal K}$.
If $\beta=-1$, then 
all classical radial stationary solutions of \eqref{Schr-Liouv} satisfying $(u,v)\in{\mathcal K}$ 
and all classical stationary solutions of system \eqref{Schr-Liouv} with $n=1$
are of the form $(c,\pm c)$, where $c\in\R$ is a constant.
Problem \eqref{Schr-bdry} does not possess nontrivial classical stationary solutions 
satisfying $(u,v)\in{\mathcal K}$ for any $\beta\in\R$.
\end{theorem} 

If $\beta\ne-1$, then
Theorem~\ref{thmSchrL}, Corollary~\ref{cor}, Theorem~\ref{thmU} with $n=1$, and Theorem~\ref{thmUhalf}
guarantee that system \eqref{Schr-Liouv} does not possess nontrivial radial solutions
satisfying $(u,v)(\cdot,t)\in{\mathcal K}$
and systems \eqref{Schr-bdry} and \eqref{Schr-Liouv} with $n=1$
do not possess nontrivial  solutions
satisfying $(u,v)(\cdot,t)\in{\mathcal K}$.
These facts together with rescaling and doubling arguments in \cite{PQS}
immediately imply the following universal $L^\infty$-estimate for radial solutions of \eqref{Schr-lambda}:

\begin{corollary} \label{corS}
Assume \eqref{ass-Schr}, $\beta\ne-1$ and $T\in(0,\infty]$. 
Then there exist $C,\tilde C\geq0$ such that any radial classical solution of \eqref{Schr-lambda}
with $(u,v)(\cdot,t)\in{\mathcal K}$
for all $t\in(0,T)$ satisfies the following estimate:
$$ \|(u,v)(\cdot,t)\|_\infty\leq \tilde C+C(t^{-1/(p-1)}+(T-t)^{-1/(p-1)}),\quad t\in(0,T),$$
where $(T-t)^{-1/(p-1)}:=0$ if $T=\infty$, and $\tilde C=0$ if $\lambda=\gamma=0$ and $\Omega=\R^n$. 
\end{corollary}

The constants $C$ and $\tilde C$ in Corollary~\ref{corS} 
depend on $\Omega,p,\beta,\gamma,\lambda$, but the dependence is
locally uniform for $\beta\in\R\setminus\{0\}$ and $\gamma,\lambda\in\R$.
The perturbation term $\tilde F(u,v):=(-\lambda u-\gamma v,-\lambda v-\gamma u)$
could be replaced with a more general perturbation which disappears
in the scaling limit. On the other hand, in order to be able to verify the assumption
$(u,v)(\cdot,t)\in{\mathcal K}$,
one should only consider perturbations which do not destroy the invariance of the set ${\mathcal K}$. 
In particular, in our case we should assume that either $\gamma=0$, or $\gamma<0$ and 
$(u,v)(\cdot,t)\in{\mathcal K}^+$.

Theorem~\ref{thmSchrL} deals
with radial solutions only, but Liouville theorems can also be true
for all (sign-changing nonradial) solutions.
In fact, assume $n\leq5$, consider the scalar problem with $F(u)=u^2$
or the case $N=2$, $F(u,v)=(2uv,u^2+v^2)$, set
$U:=u$ or $U:=(u,v)$, respectively, and notice that  the condition $F(U)\cdot U>0$
fails in these examples.
It is easily seen (see Propositions~\ref{L2} and~\ref{L2half} below) 
that the problems $-\Delta U=F(U)$ in $\R^n$ and
$-\Delta U=F(U)$ in $\R^n_+$, $U=0$ on $\partial\R^n_+$, do not
possess nontrivial solutions:

\begin{proposition} \label{L2}
Let $n\leq5$. Then the equation $-\Delta u=u^2$ and the system
$-\Delta u=2uv$, $-\Delta v=u^2+v^2$ do not possess any nontrivial
classical solution in $\R^n$.
\end{proposition}

\begin{proposition} \label{L2half}
Let $n\leq5$. Then the problems 
$$-\Delta u=u^2\quad\hbox{in }\R^n_+,\qquad u=0\quad\hbox{on }\partial\R^n_+,$$ 
and 
$$-\Delta u=2uv,\ \ -\Delta v=u^2+v^2\quad\hbox{in }\R^n_+,\qquad u=v=0\quad\hbox{on }\partial\R^n_+, $$
 do not possess any nontrivial
classical solution.
\end{proposition}

Propositions~\ref{L2}, \ref{L2half} and Theorems~\ref{thmU}, \ref{thmUhalf}
guarantee that the corresponding parabolic Liouville theorems
are true and those theorems together with scaling and doubling
arguments in \cite{PQS} guarantee
universal estimates for all global solutions 
of related initial and initial-boundary value problems. 
For simplicity, we just formulate a simple result in the scalar case;
see \cite[Theorem~3.1]{PQS} for possible modifications 
in the case of initial value problems 
without boundary conditions, for example.

\begin{corollary} \label{corScal} 
Let $n\leq5$, $T\in(0,\infty]$, $\Omega$ be a smooth domain in $\R^n$,
and $f\in C(\R)$ satisfy 
$$\lim_{|u|\to\infty}\frac{f(u)}{u^2}=\ell\in(0,\infty).$$
Then there exists $C=C(f,\Omega)$ such that any classical solution $u$
of the problem
$$ \begin{aligned}
  u_t-\Delta u &= f(u), &\qquad& x\in\Omega,\ t\in(0,T),\\
             u &= 0,  &\qquad& x\in\partial\Omega,\ t\in(0,T),
 \end{aligned}
$$
satisfies 
$$ |u(x,t)| \leq C\Bigl(1+\frac1t+\frac1{T-t}\Bigr), $$
where $1/(T-t):=0$ if $T=\infty$.
\end{corollary}

Let us mention that our parabolic Liouville theorems guarantee
universal estimates as in Corollaries~\ref{corS} and \ref{corScal}
also in the case of time-dependent perturbations $\tilde F$,
and such estimates can be used in order to
prove the existence of nontrivial periodic solutions;
see \cite{BPQ}, for example.  
Furher examples of perturbations $\tilde F$ and
applications of universal estimates can be found in
\cite{Q-DMJ,PQS}. 

\section{Proof of Theorem~\ref{thmU}} 
\label{sec-proof}

Assume to the contrary that there exists
a nontrivial solution $U$ of \eqref{eq-U} satisfying $U(\cdot,t)\in{\mathcal K}$ for all $t$.
The scaling and doubling arguments in the proof of \cite[Theorem~3.1(ii)]{PQS}
show that we may assume that 
\begin{equation} \label{bound-u}
 |U(x,t)|+|\nabla U(x,t)|\leq 1 \qquad\hbox{ for all }x\in\R^n,\ t\in\R.
\end{equation}
In fact, assume that \eqref{bound-u} fails. 
Set
$M_U:=|U|^{(p-1)/2}+|\nabla U|^{(p-1)/(p+1)}$.
Then $M_U(x_0,t_0)>2^{-(p-1)/2}$ for some $(x_0,t_0)\in\R^n\times\R$.
For any $k=1,2,\dots$, the Doubling Lemma in \cite{PQS1}
guarantees the existence of $(x_k,t_k)$ such that 
$$ \begin{aligned}
 M_k:=M_U(x_k,t_k)  &\geq M_U(x_0,t_0),\\ 
  M_U(x,t) &\leq 2M_k \quad\hbox{whenever}\quad  |x-x_k|+\sqrt{|t-t_k|}\leq \frac{k}{M_k}.
\end{aligned}
$$
The rescaled functions
$$ V_k(y,s):=\lambda_k^{2/(p-1)} U(x_k+\lambda_k y,t_k+\lambda_k^2 s),
 \quad\hbox{where}\quad \lambda_k=\frac1{2^{(p+1)/2}M_k},$$
are solutions of (\ref{eq-U}) in ${\mathcal K}$ and satisfy $M_{V_k}(0,0)=2^{-(p+1)/2}$,
$M_{V_k}(y,s)\leq2^{-(p-1)/2}$ (hence $|V_k|+|\nabla V_k|\leq1$) for $|y|+\sqrt{|s|}\leq 2^{(p+1)/2}k$. 
The parabolic regularity guarantees that the sequence $\{V_k\}$
is relatively compact (in $C_{loc}$, for example), so that
a suitable subsequence of $\{V_k\}$
converges to a solution $V$ of (\ref{eq-U}) in ${\mathcal K}$ satisfying $M_V\leq2^{-(p-1)/2}$.
Since $M_V(0,0)\ne0$, $V$ is nontrivial. 
Consequently, replacing $U$ by $V$ we may assume that (\ref{bound-u}) is true.

We may also assume $U(0,0)\ne0$.

Set $\beta:=1/(p-1)$.
By $C,C_0,C_1,\dots,c,c_0,c_1,\dots$ we will denote
positive constants which depend only on $n$, $p$ and $F$;
the constants $C,c$ may vary from step to step.
Similarly by $\eps$ and $\delta$ we will denote
small positive constants which may vary from step to step
and which only depend on $n$, $p$ and $F$.
Finally, $M=M(n,p)$ will denote a positive integer
(the number of bootstrap steps). 
The proof will be divided into several steps.

\goodbreak
{\bf Step 1: Initial estimates.}
For $a\in\R^n$ and $k=1,2,\dots$ we set
$$W(y,s)=W_k^a(y,s):=(k-t)^\beta U(y\sqrt{k-t}+a,t),\qquad\hbox{where }\  s=-\log(k-t),\ \ t<k.$$
Set also $s_k:=-\log k$ and notice that
$W=W_k^a$ solves the problem
\begin{equation} \label{eq-w}
\left.\begin{aligned}
 W_s &=\Delta W-\frac12 y\cdot\nabla W-\beta W+F(W) \\
     &=\frac1\rho\nabla\cdot(\rho\nabla W)-\beta W+F(W)\qquad \hbox{in }
\R^n\times\R,
\end{aligned}\quad\right\}
\end{equation}
where $\rho(y):=e^{-|y|^2/4}$. In addition,
$$  W_k^a(0,s_k)=k^\beta U(a,0),$$  
and
\begin{equation} \label{bound-w2}
\left.
\begin{aligned}
 \|W_k^a(\cdot,s)\|_\infty &\leq C_0k^\beta\|U(\cdot,t)\|_\infty\leq C_0k^\beta  \\
  \|\nabla W_k^a(\cdot,s)\|_\infty &\leq C_0k^{\beta+1/2}\|\nabla U(\cdot,t)\|_\infty\leq C_0k^{\beta+1/2} 
\end{aligned}
\ \right\}
\quad \hbox{for }\
s\in[s_k-M-1,\infty),
\end{equation}
where $t=k-e^{-s}$ and $C_0:=e^{(M+1)(\beta+1/2)}$. 

Set
$$ E(s)=E_k^a(s):=\frac12\int_{\R^n}\bigl(|\nabla W_k^a|^2+\beta|W_k^a|^2\bigr)(y,s)\rho(y)\,dy
 -\int_{\R^n}G(W_k^a)(y,s)\rho(y)\,dy$$
and notice that $G(W)=\frac1{p+1}F(W)\cdot W$ is $(p+1)$-homogeneous.
Then, supressing the dependence on $k,a$ in our notation, we obtain
\begin{equation} \label{GK1}
\frac12\frac{d}{ds}\int_{\R^n}|W|^2(y,s)\rho(y)\,dy  
=-(p+1)E(s)+\frac{p-1}2\int_{\R^n}(|\nabla W|^2+\beta|W|^2)(y,s)\rho(y)\,dy 
\end{equation}
and
\begin{equation} \label{GK2}
\frac{d}{ds}E(s)=-\int_{\R^n}\Big|\frac{\partial W}{\partial s}\Big|^2(y,s)\rho(y)\,dy \leq 0.
\end{equation}
Let us show that $E(s)\geq0$.
Assume to the contrary $E_0:=E(s_0)<0$ for some $s_0$ and consider $s>s_0$. Then
$$\frac12\frac{d}{ds}\int_{\R^n}|W|^2(y,s)\rho(y)\,dy\geq-(p+1)E_0+\frac12\int_{\R^n}|W|^2(y,s)\rho(y)\,dy,$$
hence $\int_{\R^n}|W|^2(y,s)\rho(y)\,dy\to\infty$ as $s\to\infty$.
However, this contradicts \eqref{bound-u}.

Finally, for any $m=1,2,\dots,M$, 
by using \eqref{GK2}, the nonnegativity and monotonicity of $E$, 
\eqref{GK1} and \eqref{bound-w2}, we obtain
\begin{equation} \label{EM}
\left.
\begin{aligned}
\int_{s_k-m}^{s_k-m+1}\int_{\R^n}& 
\Big|\frac{\partial W_k^a}{\partial s}(y,s)\Big|^2\rho(y)\,dy\,ds   
\leq E_k^a(s_k-m)    
\leq \int_{s_k-m-1}^{s_k-m} E_k^a(s)\,ds \\
&\leq \frac1{2(p+1)}\int_{\R^n}|W^a_k|^2(y,s_k-m-1)\rho(y)\,dy  \\
&\quad +\frac{p-1}{2(p+1)}\int_{s_k-m-1}^{s_k-m}\int_{\R^n}(|\nabla W^a_k|^2+\beta|W^a_k|^2)(y,s)\rho(y)\,dy\,ds \\ 
&\leq Ck^{2\beta+1}= Ck^{(p+1)\beta}.
\end{aligned}
\ \right\}
\end{equation}

{\bf Step 2: The plan of the proof.}
We will show that there exist an integer $M=M(n,p)$ and positive numbers $\gamma_m$,
$m=1,2,\dots M$, such that
$$\gamma_1<\gamma_2<\dots<\gamma_M=(p+1)\beta, \qquad
\gamma_1<\mu:=2\beta-\frac{n-2}2>0,$$
and
\begin{equation} \label{E}
E_k^a(s_k-m)\leq Ck^{\gamma_m}, \qquad  a\in\R^n,\ k\hbox{ large}, 
\end{equation}
where $m=M,M-1,\dots,1$,
and ``$k$ large'' means $k\geq k_0$ with $k_0=k_0(n,p,F,U)$. 
Then, taking $\lambda_k:=k^{-1/2}$ and setting
$$ V_k(z,\tau):=\lambda_k^{2\beta}W_k^0(\lambda_k z,\lambda_k^2\tau+s_k),
 \qquad z\in\R^n,\ -k\leq\tau\leq0,  $$ 
we obtain
$|V_k|+|\nabla V_k|\leq C$, $V_k(0,0)=U(0,0)\ne0$,
$$ \frac{\partial V_k}{\partial\tau}-\Delta V_k-F(V_k)
  =-\lambda_k^2\Bigl(\frac12 z\cdot\nabla V_k+\beta V_k\Bigr). $$
In addition,
using the first inequality in \eqref{EM} and \eqref{E} with $m=1$  we also have
\begin{equation} \label{estvtau}
\begin{aligned}
\int_{-k}^0\int_{|z|<\sqrt{k}}
 &\Big|\frac{\partial V_k}{\partial\tau}(z,\tau)\Big|^2\,dz\,d\tau
  =\lambda_k^{2\mu}
 \int_{s_k-1}^{s_k}\int_{|y|<1} 
\Big|\frac{\partial W_k^0}{\partial s}(y,s)\Big|^2\,dy\,ds \\
&\leq C k^{-\mu+\gamma_1}\to 0 \quad\hbox{as }\ k\to\infty.
\end{aligned}
\end{equation}
Now the same arguments as in \cite[pp.~18-19]{GK} 
(cf.~also \cite[Remark 3]{Q-DMJ})
show that
(up to a subsequence) the sequence $\{V_k\}$
converges to a nontrivial solution $V=V(z)\in{\mathcal K}$
of the problem $\Delta V+F(V)=0$ in $\R^n$,
which contradicts our assumption and
concludes the proof.

Notice that \eqref{E} is true if $m=M$ for any $M$ due to \eqref{EM}.
In the rest of the proof we consider $M>1$, fix $m\in\{M,M-1,\dots,2\}$,
assume that \eqref{E} is true with this fixed $m$, 
and we will prove that \eqref{E} remains true with $m$ replaced by $m-1$.
More precisely, we assume
\begin{equation} \label{Em}
E_k^a(s_k-m)\leq Ck^{\gamma}, \qquad  a\in\R^n,\ k\ \hbox{large},
\end{equation}
(where $\gamma:=\gamma_m$)
and we will show that
\begin{equation} \label{Em1}
  E_k^a(s_k-m+1)\leq Ck^{\tilde\gamma}, \qquad  a\in\R^n,\ k\ \hbox{large},
\end{equation}
where $\tilde\gamma<\gamma$ (and then we set $\gamma_{m-1}:=\tilde\gamma$).
Our proof also shows that there exists
$\delta=\delta(n,p,\gamma)\in(0,(\gamma-\tilde\gamma)/2)$ 
such that
\eqref{Em1} remains true also if \eqref{Em}
is satisfied with $\gamma$ replaced by any $\gamma'$ in the open $\delta$-neighbourhood of $\gamma$.
Since we may assume $\gamma\in[\mu,(p+1)\beta]$, 
the compactness of $[\mu,(p+1)\beta]$ guarantees that 
the difference $\gamma-\tilde\gamma$ can be bounded below by a positive constant
$\delta=\delta(n,p)$ for all $\gamma\in[\mu,(p+1)\beta]$, hence
there exists $M=M(n,p)$ such that $\gamma_1<\mu\leq\gamma_2$.  

{\bf Step 3: Notation and auxiliary results.}
In the rest of the proof we will also use the following notation
and facts:
Set 
$$C(M):=8ne^{M+1}, \ \ B_r(a):=\{x\in\R^n:|x-a|\leq r\}, \ \ B_r:=B_r(0), \ \ 
   R_k:=\sqrt{8n\log k}.$$
Given $a\in\R^n$, there exists an integer $X=X(k)$ and
there exist $a^1,a^2,\dots a^X\in\R^n$ (depending on $a,n,k$) such that
$a^1=a$, $X\leq C(\log k)^{n/2}$ and
\begin{equation} \label{B}
 D^k(a):=B_{\sqrt{C(M)k\log(k)}}(a)\subset\bigcup_{i=1}^X B_{\sqrt{k}/2}(a^i).
\end{equation}
Notice that if $y\in B_{R_k}$ and $s\in[s_k-M-1,s_k]$,
then $a+ye^{-s/2}\in D^k(a)$, hence
\eqref{B} guarantees the existence of $i\in\{1,2,\dots,X\}$
such that 
\begin{equation} \label{yyi}
W^a_k(y,s)=W^{a^i}_k(y^i,s),\qquad\hbox{where}\quad
 y^i:=y+(a-a^i)e^{s/2}\in B_{1/2}.
\end{equation}

The contradiction argument in Step~2 based on 
the nonexistence of positive stationary solutions of \eqref{eq-U}
and on estimate
\eqref{estvtau}, combined with a doubling argument,
can also be used to obtain the following
useful pointwise estimates of the solution $u$.

\begin{lemma} \label{lem-decay}
Let $M,s_k,W^a_k$ be as above, $\zeta\in\R$, $\xi,C^*>0$,
and $d_k,r_k\in(0,1]$, $k=1,2,\dots$.
Set
$$ {\mathcal T}_k:=\Bigl\{(a,\sigma,b)\in\R^n\times(s_k-M,s_k]\times\R^n:
  \int_{\sigma-d_k}^{\sigma}\int_{B_{r_k}(b)} |(W^a_k)_s|^2 dy\,ds \leq C^*k^\zeta
\Bigr\}. $$ 
Assume
\begin{equation} \label{xizeta}
 \xi\frac\mu\beta>\zeta\quad{and}\quad
 \frac1{\log(k)}\min(d_kk^{\xi/\beta},r_kk^{\xi/2\beta})\to\infty\ 
 \hbox{ as }\ k\to\infty.
\end{equation}
Then there exists $\tilde k_0$ such that
$$ (|W^a_k|+|\nabla W^a_k|^{2/(p+1)})(y,\sigma)\leq k^\xi \quad\hbox{whenever}\quad
 y\in B_{r_k/2}(b), \ \  k\geq \tilde k_0 \ 
\hbox{ and } \ (a,\sigma,b)\in{\mathcal T}_k.$$ 
\end{lemma}

\begin{proof}
The proof is a straightforward modification of the proof of
\cite[Lemma~6]{Q-DMJ} (cf.~also \cite[ Remark~7]{Q-DMJ}).
For the reader's convenience we provide it in detail.
Assume to the contrary that there exist $k_1,k_2\dots$ with the following properties:
$k_j\to\infty$ as $j\to\infty$, and for each $k\in\{k_1,k_2,\dots\}$ there exist
$(a_k,\sigma_k,b_k)\in{\mathcal T}_k$ and $y_k\in B_{r_k/2}(b_k)$ such that
$\tilde w_k(y_k,\sigma_k)>k^\xi$, where 
$\tilde w_k:=|W^{a_k}_k|+|\nabla W^{a_k}_k|^{2/(p+1)}$.

Given $k\in\{k_1,k_2\,\dots\}$, we can choose an integer $K$ such that
\begin{equation} \label{Klemma}
  2^Kk^{\xi}>2C_0k^\beta, \qquad K<C\log k.
\end{equation}
Set
$$ Z_q:=B_{r_k(1/2+q/(2K))}(b_k)\times[\sigma_k-d_k(1/2+q/(2K)),\sigma_k],
 \quad q=0,1,\dots,K.$$
Then
$$ B_{r_k/2}(b_k)\times[\sigma_k-d_k/2,\sigma_k]=Z_0\subset Z_1\subset\dots\subset Z_K=B_{r_k}(b_k)\times[\sigma_k-d_k,\sigma_k].$$
Since $\sup_{Z_0}\tilde w_k\geq \tilde w_k(y_k,\sigma_k)>k^{\xi}$,
estimates \eqref{Klemma} and \eqref{bound-w2} imply the existence of $q^*\in\{0,1,\dots K-1\}$ such that
$$ 2\sup_{Z_{q^*}}\tilde w_k \geq \sup_{Z_{q^*+1}}\tilde w_k $$
(otherwise $2C_0k^\beta\geq\sup_{Z_K}\tilde w_k>2^K\sup_{Z_0}\tilde w_k>2^K k^\xi$, a contradiction).
Fix $(\hat y_k,\hat s_k)\in Z_{q^*}$ such that
$$ \tilde W_k:=\tilde w_k(\hat y_k,\hat s_k)=\sup_{Z_{q^*}}\tilde w_k.$$
Then $\tilde W_k\geq  k^{\xi}$,
$$ \hat Q_k:=B_{r_k/(2K)}(\hat y_k)\times\Bigl[\hat s_k-\frac{d_k}{2K},\hat s_k\Bigr]\subset Z_{q^*+1},$$
and $\tilde w_k\leq 2\tilde W_k$ on $\hat Q_k$.

Set $\lambda_k:=\tilde W_k^{-1/(2\beta)}$ (hence $\lambda_k\leq k^{-\xi/(2\beta)}\to 0$ as $k\to\infty$)
and
$$ V_k(z,\tau):=\lambda_k^{2\beta}W^{a_k}_k(\lambda_k z+\hat y_k,\lambda_k^2\tau+\hat s_k).$$
Then $V_k(\cdot,\tau)\in{\mathcal K}$,
$(|V_k|+|\nabla V_k|^{2/(p+1)})(0,0)=1$, $|V_k|+|\nabla V_k|^{2/(p+1)}\leq2$ 
on $Q_k:=B_{r_k/(2K\lambda_k)}\times[-d_k/(2K\lambda_k^2),0]$, and
\begin{equation} \label{eqvk} 
  \frac{\partial V_k}{\partial\tau}-\Delta V_k-F(V_k)
  =-\lambda_k^2\Bigl(\frac12 z\cdot\nabla V_k+\beta V_k\Bigr) \quad\hbox{on }\ Q_k. 
\end{equation}
In addition, as $k\to\infty$,
$$
  \frac{r_k}{2K\lambda_k} \geq\frac{r_k k^{\xi/(2\beta)}}{C\log(k)}\to\infty, \quad
  \frac{d_k}{2K\lambda_k^2} \geq\frac{d_k k^{\xi/\beta}}{C\log(k)}\to\infty.
$$
Since $(a_k,\sigma_k,b_k)\in{\mathcal T}_k$ and $\hat Q_k\subset Z_K$, we obtain
$$
\int_{Q_k}\Big|\frac{\partial V_k}{\partial\tau}(z,\tau)\Big|^2\,dz\,d\tau
  =\lambda_k^{2\mu}\int_{\hat Q_k}
\Big|\frac{\partial W^{a_k}_k}{\partial s}(y,s)\Big|^2\,dy\,ds \leq C^*k^\delta,
\quad\hbox{where }\ \delta:=-\xi\frac\mu\beta+\zeta <0.
$$
Hence, as above, a suitable subsequence of $\{V_k\}$
converges to a nontrivial solution $V\in{\mathcal K}$
of the problem $\Delta V+F(V)=0$ in $\R^n$,
which contradicts our assumptions.
\end{proof}

Recall that $\gamma\in[\mu,(p+1)\beta]$.

\begin{Lemma} \label{lemmaG}
Let ${\mathcal T}_k={\mathcal T}_k(d_k,r_k,\zeta,C^*)$ be as in Lemma~\ref{lem-decay},
$\omega\in\R$, $\eps,C^*>0$,
$$ 0\leq\alpha<\frac\xi\beta,\qquad \xi\frac\mu\beta>\gamma-\alpha+\eps-\omega, $$ 
and assume that 
\begin{equation} \label{assG}
 \hbox{$(a,\sigma,0)\in {\mathcal T}_k(\frac12 k^{-\alpha},1,\gamma-\alpha+\eps,C^*)$\ \ for $k$ large}.
\end{equation}
Set 
$$G:=\{y\in B_{1/2}: (|W^a_k|+|\nabla W^a_k|^{2/(p+1)})(y,\sigma)\leq k^\xi\}.$$
Then 
\begin{equation} \label{Gc}
|B_{1/2}\setminus G|\leq Ck^{\omega-n\alpha/2}\ \hbox{ for $k$ large}.
\end{equation}
\end{Lemma}

\begin{proof}
There exist 
$b^1,\dots,b^{Y}\in B_{1/2}$ with $Y\leq Ck^{n\alpha/2}$
such that 
$$ B_{1/2}\subset \bigcup_{j=1}^{Y}B^j,\quad\hbox{where}\quad
  B^j:=B_{\frac12k^{-\alpha/2}}(b^j), $$
and 
\begin{equation} \label{multiynew}
 \#\{j: y\in B_{k^{-\alpha/2}}(b^j)\}\leq C_n\quad\hbox{for any }\ y\in\R^n.
\end{equation}
Set
$$ \begin{aligned}
  H &:=\bigl\{j\in\{1,2,\dots,Y\}: 
  (a,\sigma,b^j)\in{\mathcal T}_k(\hbox{$\frac12$}k^{-\alpha},k^{-\alpha/2},\gamma-\alpha+\eps-\omega,C^*C_n)\bigr\}, \\
  H^c &:=\{1,2,\dots,Y\}\setminus H.
\end{aligned}
$$
If $j\in H$, then  Lemma~\ref{lem-decay} 
guarantees $(|W^a_k|+|\nabla W^a_k|^{2/(p+1)})(y,\sigma)\leq k^\xi$ for $y\in B^j$.
Consequently, 
$$B_{1/2}\cap\bigcup_{j\in H}B^j \subset G, \ \hbox{ hence } \
  B_{1/2}\setminus G \subset \bigcup_{j\in H^c}B^j. $$
Now \eqref{assG}, the definition of $H$ and \eqref{multiynew} imply
$\#H^c< k^\omega$, 
hence \eqref{Gc} is true.
\end{proof}

\begin{Lemma} \label{lem-ineq}
Fix a positive integer $L=L(n,p)$ such that
\begin{equation} \label{estL}
 \frac\mu{p+1}\Bigl(\frac{2p}{p+1}\Bigr)^L>\beta.
\end{equation}
If $\eps,\delta>0$ are small enough, then there exist
$\xi_\ell,\alpha_\ell,\omega_\ell$, $\ell=1,2,\dots L$, 
such that 
\begin{equation} \label{estxi}
\xi_1\leq\xi_2\leq\dots\leq\xi_L\leq\beta=:\xi_{L+1},
\end{equation}
$(p+1)\xi_1\leq\gamma-\delta$,
and the following inequalities are true for $\ell=1,2,\dots,L$:
\begin{equation} \label{bootcond}
0\leq\alpha_\ell <\frac{\xi_\ell}\beta, \quad \xi_\ell\frac\mu\beta>\gamma-\alpha_\ell+\eps-\omega_\ell, \quad
\omega_\ell-\frac{n\alpha_\ell}2\leq\gamma-\delta-(p+1)\xi_{\ell+1}.
\end{equation}
\end{Lemma}

\begin{proof}
Considering $\alpha_\ell$ close to (and less than) $\xi_\ell/\beta$ and $\eps,\delta>0$ small,
we see that we only have to satisfy the condition $(p+1)\xi_1<\gamma$ and the following inequalities
for $\ell=1,2,\dots,L$:
\begin{equation} \label{bootcond2}
 \xi_\ell\frac\mu\beta>\gamma-\frac{\xi_\ell}\beta-\omega_\ell, \quad
\omega_\ell-\frac{n}2\frac{\xi_\ell}\beta<\gamma-(p+1)\xi_{\ell+1}.
\end{equation}
One can find $\omega_\ell$ such that the inequalities \eqref{bootcond2} are true,
provided the lower bound for $\omega_\ell$ in these inequalities
is less than the upper bound, i.e. if
$$ \gamma-\frac{\xi_\ell}\beta(1+\mu)<\frac{n}2\frac{\xi_\ell}\beta+\gamma-(p+1)\xi_{\ell+1}, $$
which is equivalent to
$$ (p+1)\xi_{\ell+1}<2p\xi_\ell. $$
Consequently, choosing $\xi_1$ close to (and less than) $\gamma/(p+1)$ and using
$\gamma\in[\mu,(p+1)\beta]$, the existence
of $\xi_2,\dots,\xi_L$ follows.
\end{proof}

\begin{Lemma} \label{lem-wq}
Let $L,\eps,\delta$ and 
$\xi_\ell,\alpha_\ell,\omega_\ell$, $\ell=1,2,\dots L$, be as in
Lemma~\ref{lem-ineq} and let
${\mathcal T}_k={\mathcal T}_k(d_k,r_k,\zeta,C^*)$ be as in Lemma~\ref{lem-decay}.
Assume $(a,\sigma,0)\in {\mathcal T}_k(\frac12 k^{-\alpha_\ell},1,\gamma-\alpha_\ell+\eps,C)$
for $\ell=1,2,\dots L$ and $k$ large. 
Then 
\begin{equation} \label{est-wq}
\int_{B_{1/2}}(|\nabla W^a_k|^2+|W^a_k|^2+|W^a_k|^{p+1})(y,\sigma)\,dy\leq Ck^{\gamma-\delta}\ \hbox{ for $k$ large}.
\end{equation}
\end{Lemma}

\begin{proof}
Given $\ell\in\{1,2,\dots,L\}$, set $\xi=\xi_\ell$, $\alpha=\alpha_\ell$, $\omega=\omega_\ell$,
and let $G$ be the set in Lemma~\ref{lemmaG}.
Set $G_\ell:=G$ and $G_{L+1}:=B_{1/2}$. 
Lemma~\ref{lemmaG} and \eqref{bootcond} guarantee
$$|G_{\ell+1}\setminus G_\ell|\leq|B_{1/2}\setminus G_\ell|\leq Ck^{\omega_\ell-n\alpha_\ell/2}
                 \leq Ck^{\gamma-\delta-(p+1)\xi_{\ell+1}},$$
hence
$$\int_{G_{\ell+1}\setminus G_\ell}(|\nabla W^a_k|^2+|W^a_k|^2+|W^a_k|^{p+1})(y,\sigma)\,dy
 \leq Ck^{(p+1)\xi_{\ell+1}}|G_{\ell+1}\setminus G_\ell| \leq Ck^{\gamma-\delta}.$$
In addition, the definition of $G_1$ implies
\begin{equation} \label{estG1}
 \int_{G_1}(|\nabla W^a_k|^2+|W^a_k|^2+|W^a_k|^{p+1})(y,\sigma)\,dy 
 \leq Ck^{(p+1)\xi_1} \leq Ck^{\gamma-\delta}.
\end{equation}
Since $B_{1/2}=G_1\cup\bigcup_{\ell=1}^L (G_{\ell+1}\setminus G_\ell)$,
the conclusion follows.
\end{proof}

\smallskip\noindent{\bf Step 4: The choice of a suitable time.}
The proof of \eqref{Em1} will be based on estimates of $W^{a^i}_k(\cdot,s^*)$,
$i=1,2,\dots,X$, where
$s^*=s^*(k,a)\in[s_k-m,s_k-m+1]$ is a suitable time. 
The following lemma is a special case of \cite[Lemma 8]{Q-JDDE}.

\begin{Lemma} \label{lem-step4}  
Let $\eps,\gamma,C_1>0$, $\alpha_1,\alpha_2,\dots,\alpha_L\geq0$,
and, given $k=1,2,\dots$, let $X_k$ be a positive integer satisfying 
$X_k\leq k^{\eps/2}$ and $\sigma_k\in\R$. Set $J_k:=[\sigma_k,\sigma_k+1]$, 
$\tilde J_k:=[\sigma_k+1/2,\sigma_k+1]$, and
assume that  $f^1_k,\dots,f^{X_k}_k\in C(J_k,\R^+)$ 
satisfy 
\begin{equation} \label{figi}
\int_{J_k}f^i_k(s)\,ds\leq C_1k^\gamma, \quad 
i=1,2,\dots X_k,\ k=1,2,\dots.
\end{equation}
Then there exists $k_1=k_1(\eps,L)$ with the following property:
If $k\geq k_1$, then there exists $s^*=s^*(k)\in\tilde J_k$ such that 
$$ \int_{s^*-\frac12k^{-\alpha_\ell}}^{s^*}f^i_k(s)\,ds\leq C_1k^{\gamma-\alpha_\ell+\eps} $$ 
for all $i=1,2,\dots,X_k$ and $\ell=1,2,\dots,L$.
\end{Lemma}

\begin{proof}
Set
$$h^{i,\ell}_k(s):=\int_{s-\frac12k^{-\alpha_\ell}}^{s}f^i_k(\tau)\,d\tau, \ \ s\in\tilde J_k,\ \
i=1,2,\dots X_k,\ \ \ell=1,2,\dots,L,\ \ k=1,2,\dots.$$
Then 
\begin{equation} \label{hiell}
\begin{aligned}
\int_{\tilde J_k} &h^{i,\ell}_k(s)\,ds
  = \int_{\tilde J_k}\int_{s-\frac12k^{-\alpha_\ell}}^{s}f^i_k(\tau)\,d\tau\,ds
  = \int_{\tilde J_k}\int_0^{\frac12k^{-\alpha_\ell}}f^i_k(s-\tau)\,d\tau\,ds \\
 &= \int_0^{\frac12k^{-\alpha_\ell}}\int_{\tilde J_k}f^i_k(s-\tau)\,ds\,d\tau
  \leq \int_0^{\frac12k^{-\alpha_\ell}}\int_{J_k} f^i_k(s)\,ds\,d\tau\leq C_1k^{\gamma-\alpha_\ell}.
\end{aligned}
\end{equation}
Set
$$ \begin{aligned}
   C^{i,\ell}_k &:=\{s\in\tilde J_k: h^{i,\ell}_k(s)>C_1k^{\gamma-\alpha_\ell+\eps}\}.
\end{aligned}$$
Then \eqref{hiell} implies $|C^{i,\ell}_k|\leq k^{-\eps}$ for each $i,\ell,k$.
Since the number of sets $C^{i,\ell}_k$ with fixed index $k$ is $LX_k\leq Lk^{\eps/2}$,
their union
$U_k:=\bigcup_{i,\ell}C^{i,\ell}_k$
has measure less than 1/2 for $k\geq k_1$, hence for $k\geq k_1$ 
there exists $s^*=s^*(k)\in\tilde J_k\setminus U_k$.
Obviously, $s^*$ has the required properties.
\end{proof}

Consider $m$, $\gamma\in[\mu,(p+1)\beta]$ and $a\in\R^n$ fixed,
$J_k:=[s_k-m,s_k-m+1]$,
and let $a^i$, $i=1,2,\dots,X$ be as in \eqref{B} (recall that $a^i$ and $X$ depend on $k$;
$X\leq C(\log k)^{n/2}$).
Let $L,\eps$ and $\alpha_\ell$, $\ell=1,2,\dots,L$ be as in Lemma~\ref{lem-ineq}.
Set 
$$f^i_k(s):=\int_{\R^n}|(W^{a^i}_k)_s|^2(y,s)\rho(y)\,dy,\quad
\quad i=1,2,\dots X.$$
Then \eqref{Em} and the first inequality in \eqref{EM} guarantee
that the assumptions of  Lemma~\ref{lem-step4} are satisfied
with $C_1$ independent of $a$.
Consequently, if $k\geq k_1$, then
there exists $s^*=s^*(k,a)\in\tilde J_k:=[s_k-m+1/2,s_k-m+1]$ such that the following estimates are true  
for $a\in\R^n$, $W=W_k^{a^i}$,
$i=1,2,\dots X$, $\ell=1,2,\dots L$: 
\begin{equation} \label{star1}
 \int_{s^*-\frac12k^{-\alpha_\ell}}^{s^*}\int_{\R^n}|W_s|^2(y,s)\rho(y)\,dy\,ds \leq C_1k^{\gamma-\alpha_\ell+\eps}.
\end{equation}

\smallskip\noindent{\bf Step 5: Energy estimates.}
Let $L,\eps,\delta$ and 
$\xi_\ell,\alpha_\ell,\omega_\ell$, $\ell=1,2,\dots L$, be as in
Lemma~\ref{lem-ineq} and let
${\mathcal T}_k={\mathcal T}_k(d_k,r_k,\zeta,C^*)$ be as in Lemma~\ref{lem-decay}.
Let $a\in\R^n$ be fixed and $s^*=s^*(k,a)$ be from Step~4.
Notice that \eqref{star1} guarantees 
$(a^i,s^*,0)\in{\mathcal T}_k(\frac12k^{-\alpha_\ell},1,\gamma-\alpha_\ell+\eps,C_1/\rho(1))$
for $i=1,2,\dots,X$ and $\ell=1,2,\dots,L$.
Consequently, Lemma~\ref{lem-wq} implies
$$  \int_{B_{1/2}}(|\nabla W^{a^i}_k|^2+|W^{a^i}_k|^2+|W^{a^i}_k|^{p+1})(y,s^*)\,dy
\leq Ck^{\gamma-\delta}\   \hbox{ for $i=1,2,\dots,X$ and $k$ large}, $$ 
and using \eqref{yyi} we obtain
$$ \begin{aligned}
\int_{B_{R_k}}&(|\nabla W^a_k|^2+|W^a_k|^2+|W^a_k|^{p+1})(y,s^*)\,dy \\
 &\leq \sum_{i=1}^X\int_{B_{1/2}}(|\nabla W^{a^i}_k|^2+|W^{a^i}_k|^2+|W^{a^i}_k|^{p+1})(y,s^*)\,dy \\
 &\leq Ck^{\gamma-\delta}(\log k)^{n/2}
 \leq Ck^{\gamma-\delta/2}\quad  \hbox{ for $k$ large}.
\end{aligned} $$ 
In addition, since
$$  \rho(y)=e^{-|y|^2/8-|y|^2/8}\leq k^{-n}e^{-|y|^2/8}\quad\hbox{for }\ |y|>R_k,$$
we have
$$\begin{aligned}
\int_{\R^n\setminus B_{R_k}}&(|\nabla W^a_k|^2+|W^a_k|^2+|W^a_k|^{p+1})(y,s^*)\rho(y)\,dy \\
 &\leq C\int_{\R^n\setminus B_{R_k}}k^{(p+1)\beta-n}e^{-|y|^2/8}\,dy 
 \leq Ck^{(p+1)\beta-n} = Ck^{\mu-n/2} \leq C k^{\gamma-\delta/2} 
\end{aligned}
$$
due to $\gamma\geq\mu$,
hence
$E^a_k(s^*)\leq Ck^{\gamma-\delta/2}$ for $k$ large,
and the monotonicity of $E_k^a$ implies \eqref{Em1}.
This concludes the proof.
\qed

\section{Proof of Theorem~\ref{thmUhalf}} 
\label{sec-half}

Using the arguments in \cite{GK}, the proof is more or less
straightforward modification of the proof of Theorem~\ref{thmU}.
Let us just mention the main differences:

Step 1: We may assume that $U(a_0,0)\ne0$, where $a_0:=(e,0,0,\dots,0)$.
The rescaled function $W=W^a_k$ (defined for $a\in\overline{\R^n_+}$) is a solution in
$\bigcup_s\Omega_a(s)\times\{s\}$, where $\Omega_a(s):=e^{s/2}(\R^n_+-a)$,
and satifies the boundary conditions
$W=W_s+\frac12 y\cdot\nabla W=0$ on $\bigcup_s\partial\Omega_a(s)\times\{s\}$
(see \cite{GK}).
Formula \eqref{GK2} has the form
$$
\frac{d}{ds}E(s)=-\int_{\Omega_a(s)}\Big|\frac{\partial W}{\partial s}\Big|^2(y,s)\rho(y)\,dy
-\frac14\int_{\partial\Omega_a(s)}(y\cdot\nu)\Big|\frac{\partial W}{\partial\nu}\Big|^2(y,s)\rho(y)\,dS_y,
$$
where $\nu:=(-1,0,0,\dots,0)$ and $y\cdot\nu\geq0$ (see \cite{GK}).

Step 2: In the definition of $V_k$ we rescale the function $W^{a_0}_k$ (instead of $W^0_k$).
The limit $V$ is again a nontrivial solution in $\R^n$. 

Step 3: 
The main differences are in Lemma~\ref{lem-decay}, where one has to work
in balls $B_{r_k}(b)$ intersected with $\Omega_a(s)$
and use regularity estimates based on the arguments in \cite[pp.~15--16]{GK}.

The modifications in the rest of the proof are straightforward.
\qed

\section{Proof of Propositions~\ref{L2}, \ref{L2half}}  
\label{sec-Schr}

Assume first to the contrary that $u$ is a nontrivial classical solution 
of $-\Delta u=u^2$ in $\R^n$, $n\leq5$.
Doubling and scaling arguments guarantee that we may assume that $u$ is bounded
(cf.~the beginning of the proof of Theorem~\ref{thmU}).
The Liouville theorem in \cite{GS} guarantees that $u$ is not nonnegative,
and we may assume $u(0)=:c_0<0$.
Fix $\eps\in(0,c_0^2/(2n))$ and set $u_\eps(x):=u(x)+\eps|x|^2$.
Notice that $u_\eps(x)\to\infty$ as $|x|\to\infty$.
so that we can find $x_\eps$ such that 
$$0>c_0=u_\eps(0)\geq \min u_\eps=u_\eps(x_\eps)\geq u(x_\eps).$$
Then 
$$ 0\geq-\Delta u_\eps(x_\eps)=u^2(x_\eps)-2\eps n\geq c_0^2-2\eps n>0,$$
which yields a contradiction.

Next assume that $(u,v)$ is a nontrivial classical solution
of the system $-\Delta u=2uv$, $-\Delta v=u^2+v^2$ in $\R^n$.
Then $w:=u+v$ solves $-\Delta w=w^2$, hence $w=0$, $u=-v$.
By setting $u=-v$ in the equation $-\Delta u=2uv$ we obtain $-\Delta v=2v^2$,
hence $v=0$, $u=0$, which yields a contradiction.

The proof of Proposition~\ref{L2half} is a straightforward modification
of the above proof: Instead of the Liouville theorem in \cite{GS}
one can use the Liouville theorem in \cite{GS1},
instead of the assumption $u(0)=:c_0<0$
we may assume that $u(x_0)=:c_0<0$ for some $x_0\in\R^n_+$
and then set $u_\eps(x):=u(x)+\eps|x-x_0|^2$ and observe that $x_\eps$ can be found in $\R^n_+$.
\qed

\begin{remark} \rm
Let $\delta>0$, $n\leq4$, and assume that the system
\begin{equation} \label{uv}
-\Delta u=2uv,\qquad -\Delta v=u^2+\delta v^2
\end{equation}
possesses a nontrivial solution in $\R^n$.
Then the inequality $-\Delta v\geq\delta v^2$ and \cite{BVP}
guarantee that either $v\equiv0$ (which leads to a contradiction) 
or $v$ cannot be nonnegative. We may assume $v(0)<0$.
Setting $v_\eps(x):=v(x)+2\eps|x|^2$, the same argument as in the proof
of Proposition~\ref{L2} yields a contradiction.

If $\delta=0$, then system \eqref{uv} possesses
constant solutions of the form $(0,c)$.
On the other hand, \cite{RZ} guarantees that such constant solutions with $c>0$
are the only nontrivial nonnegative solutions if $n\leq5$.
See also \cite{ZZS}, for example, for results on the related system
$-\Delta u+u=2uv$, $-\Delta v+\alpha v=u^2$, where $\alpha>0$.
\end{remark}


\end{document}